\documentclass[11pt]{amsart}
\usepackage[margin=1.4in]{geometry}
\usepackage{amsfonts,amsmath,amsthm,amscd,amsxtra}
\usepackage{mathtools,mathptmx}
\usepackage{enumerate,epsfig,verbatim}
\usepackage{centernot}
\usepackage{todonotes}
\usepackage[T1]{fontenc}
\usepackage[utf8]{inputenc}
\usepackage[usenames,dvipsnames]{pstricks}
\usepackage[mathscr]{eucal}
\usepackage{tikz}
\usepackage{amssymb}
\usepackage{mathdots}
\usepackage{mathrsfs}  
\usetikzlibrary{graphs,graphs.standard,matrix, calc, arrows,quotes}
\usepackage[pagebackref]{hyperref}
\usepackage{nicefrac}
\usepackage{array}
\usepackage{xcolor}
\usepackage{color}
\usepackage[all]{xypic}
\usepackage[notcite,notref]{}
\usepackage{url}
\usepackage{datetime}

\SelectTips{cm}{}

\def\latex/{{\protect\LaTeX}}
\def\latexe/{{\protect\LaTeXe}}
\def\amslatex/{{\protect\AmS-\protect\LaTeX}}
\def\tex/{{\protect\TeX}}
\def\amstex/{{\protect\AmS-\protect\TeX}}
\def\bibtex/{{Bib\protect\TeX}}
\def\makeindx/{\textit{MakeIndex}}
\theoremstyle{plain} 
\newtheorem{thm}{Theorem}[section]
\newtheorem{prop}[thm]{Proposition}
\newtheorem{cor}[thm]{Corollary}

\theoremstyle{definition}
\newtheorem{chunk}[thm]{\hspace*{-1.065ex}\bf}
\newtheorem{lem}[thm]{Lemma}

\newtheorem{eg}[thm]{Example}

\newtheorem{rmk}[thm]{Remark}
\newtheorem{diss}[thm]{Discussion}

\theoremstyle{remark}

\newtheorem*{claim*}{Claim}

\newcommand{\ZZ}{\mathbb{Z}}

\newcommand{\fm}{\mathfrak{m}}

\newcommand{\fq}{\mathfrak{q}}
\DeclareMathOperator{\e}{\operatorname{{e}}}

\DeclareMathOperator{\id}{id}

\DeclareMathOperator{\Tor}{Tor}
\DeclareMathOperator{\Ext}{Ext}

\DeclareMathOperator{\Hom}{Hom}

\DeclareMathOperator{\len}{length}

\DeclareMathOperator{\im}{im}

\DeclareMathOperator{\pd}{pd}

\DeclareMathOperator{\edim}{embdim}

\newcommand{\Ann}{\textup{Ann}}

\newcommand{\bb}{\left[ \begin{smallmatrix}}
\newcommand{\eb}{\end{smallmatrix} \right]}
\def\urltilda{\kern -.15em\lower .7ex\hbox{\~{}}\kern
  .04em}\def\urldot{\kern -.10em.\kern -.10em}\def\urlhttp{http\kern
  -.10em\lower -.1ex\hbox{:}\kern -.12em\lower 0ex\hbox{/}\kern
  -.18em\lower 0ex\hbox{/}}

\newcommand{\rmQ}{\mathrm{Q}}
\newcommand{\Q}{\mathrm{Q}}

\newcommand{\fkq}{\mathfrak{q}}

\newcommand{\con}{\mathfrak{c}}

\def\urltilda{\kern -.15em\lower .7ex\hbox{\~{}}\kern .04em}
\def\urldot{\kern -.10em.\kern -.10em}\def\urlhttp{http\kern -.10em\lower -.1ex
\hbox{:}\kern -.12em\lower 0ex\hbox{/}\kern -.18em\lower 0ex\hbox{/}}

\title[Modules of minimal multiplicity]{Modules of minimal multiplicity over One-Dimensional Cohen-Macaulay Local Rings}

\author[Ela Celikbas]{Ela Celikbas}
\address{Ela Celikbas\\
School of Mathematical and Data Sciences\\
West Virginia University \\
Morgantown, WV 26506-6310, USA}
\email{ela.celikbas@math.wvu.edu}

\author[Olgur Celikbas]{Olgur Celikbas}
\address{Olgur Celikbas\\
School of Mathematical and Data Sciences\\
West Virginia University \\
Morgantown, WV 26506-6310, USA}
\email{olgur.celikbas@math.wvu.edu}

\author[Naoki Endo]{Naoki Endo}
\address{Naoki Endo\\
School of Political Science and Economics, Meiji University, 1-9-1, Eifuku, Suginami-ku, Tokyo, 168-8555, Japan} 
\email{endo@meiji.ac.jp}

\author[Shinya Kumashiro]{Shinya Kumashiro}
\address{Shinya Kumashiro\\
Faculty of Engineering, Osaka Institute of Technology
5-16-1 Omiya, asahi-ku, Osaka, 535-8585, Japan}
\email{shinya.kumashiro@oit.ac.jp}


\begin{document}

\begin{abstract} We study finitely generated modules of minimal multiplicity, a notion introduced by Puthenpurakal that extends the classical concept of minimal multiplicity from rings to modules. Our main result characterizes when trace ideals or reflexive ideals yield modules of minimal multiplicity over one-dimensional Cohen-Macaulay local rings. As a consequence, we show that a one-dimensional non-Gorenstein reduced local ring with a canonical module has minimal multiplicity if and only if its canonical module has minimal multiplicity as a module. We also construct several examples and compare them with Burch and Ulrich modules, highlighting cases where minimal multiplicity coincides with the Burch or Ulrich property.
\end{abstract}

\maketitle

\section{Introduction}

Throughout this paper, unless otherwise stated, $R$ denotes a \emph{one-dimensional commutative Noetherian Cohen-Macaulay local ring} with maximal ideal $\fm$ and residue field $k$. Moreover, all $R$-modules are assumed to be finitely generated.

In this paper we are concerned with modules of minimal multiplicity, a class that was initially defined and studied by Puthenpurakal~\cite{TonyP} over Cohen-Macaulay rings of arbitrary dimension. Various properties of modules of minimal multiplicity have been more recently investigated in~\cite{Souvik2} and~\cite{Souvik1}.

Our main purpose is to investigate modules of minimal multiplicity with respect to ideals that are either trace or reflexive. Note that, given a ring $R$, an $\fm$-primary ideal $I$ of $R$ is called \emph{reflexive} if $\left(R:_{\rm Q}(R:_{\rm Q}I)\right)=I$, and called \emph{trace} if $\left(I:_{\rm Q}I\right)=\left(R:_{\rm Q}I\right)$, where $\rm Q$ denotes the total ring of fractions of $R$. In this direction, we prove the following: 

\begin{thm} \label{thm-intro} Let $R$ be a one-dimensional generically Gorenstein non-Gorenstein ring with canonical ideal $\omega_R$, and let $I$ be an $\mathfrak{m}$-primary ideal of $R$. Assume $I$ is either a trace ideal of $R$, or reflexive as an $R$-module. Then the following conditions are equivalent:
\begin{enumerate}[\rm(i)]
\item $I$ is stable, that is, $I^2=xI$ for some $x\in I$.
\item Every torsion-free $R$-module has minimal multiplicity with respect to $I$.
\item $\omega_R$ has minimal multiplicity with respect to $I$.
\end{enumerate}
\end{thm}

Theorem~\ref{thm-intro} is a consequence of a more general result, namely Theorem~\ref{corchar1}, and is proved as Corollary~\ref{corcharI}; see also Remark~\ref{SI}(i) and~\ref{fractional} for related details.

A (one-dimensional) Cohen-Macaulay local ring is said to have minimal multiplicity if its multiplicity equals its embedding dimension~\cite{Abhyankar}. Using Theorem~\ref{thm-intro} together with other results established in this paper, we characterize the minimal multiplicity of rings in terms of that of certain modules. Our result in this direction can be summarized as follows; see Proposition~\ref{prop-min-mul}, Lemma~\ref{m-dual}, and Corollary~\ref{corchar2} for details.

\begin{cor} \label{cor-intro} Let $R$ be a one-dimensional generically Gorenstein non-Gorenstein ring with canonical ideal $\omega_R$. Then the following conditions are equivalent:
\begin{enumerate}[\rm(i)]
\item $R$ has minimal multiplicity.
\item The idealization ring $R \ltimes \fm$ has minimal multiplicity.
\item Every torsion-free $R$-module has minimal multiplicity.
\item $\omega_R$ has minimal multiplicity.
\item The endomorphism algebra $\Hom_R(\fm, \fm)$ has minimal multiplicity as an $R$-module.
\end{enumerate}
\end{cor}

Now we briefly discuss the organization of the paper. In Section~\ref{D-M}, we give several examples of modules of minimal multiplicity to illustrate various properties. For example, we use the conductor of the ring to produce new examples, and exhibit a module $M$ that does not have minimal multiplicity, while its submodule $\fm M$ does; see Examples~\ref{Arfmax} and~\ref{e1}. We provide a characterization of the minimal multiplicity of idealizations and observe that (co)syzygies of modules of minimal multiplicity do not necessarily have minimal multiplicity; see Proposition~\ref{prop-min-mul} and Example~\ref{Arfsyz}. It is known that every Ulrich module has minimal multiplicity; see, for example,~\cite[2.8]{Souvik1}. On the other hand, there are modules of minimal multiplicity that are not Ulrich; see Example~\ref{minmulnotulrich}. As a related result, we show in Proposition~\ref{reduction} that, if $r$ is the reduction number of a given non-principal $\fm$-primary ideal $I$, then $I^{r-1}$ has minimal multiplicity with respect to $I$, but it is not an $I$-Ulrich $R$-module. We also compare the class of modules of minimal multiplicity with that of Burch modules and corroborate the fact that these classes differ in general; see~\ref{Burch-facts}. 

Section~\ref{section-MMB} is devoted to one of our contributions, namely, identifying certain classes of modules and ideals for which the Ulrich property and minimal multiplicity are equivalent; see Theorem~\ref{AGchar}. In Section~\ref{main-section} we give a proof of Theorem~\ref{thm-intro}; see also Theorem~\ref{corchar1}. In~\ref{Aseq} of the appendix, we establish an exact sequence involving modules of minimal multiplicity, which appears to be new. Using this exact sequence, we provide an alternative proof of the characterization of modules of minimal multiplicity stated in~\ref{corchar}. A second application of the same exact sequence, concerning the vanishing of certain $\Tor$ modules, is presented in~\ref{cx}.

\section{Definitions, Examples, and Remarks} \label{C-E-R}  \label{D-M}

We proceed by recalling our setup:

\begin{chunk} \label{setup} Throughout, unless otherwise stated,  $R$ denotes a \emph{one-dimensional commutative Noetherian Cohen-Macaulay local ring} with maximal ideal $\fm$ and residue field $k$. Moreover, we assume all $R$-modules are finitely generated. We also assume that there is always a minimal reduction of every $\fm$-primary ideal $I$ of $R$; this holds, for example, if the residue field $k$ of $R$ is infinite~\cite[8.3.7]{HunekeSwanson}. 
\end{chunk}

Puthenpurakal~\cite{TonyP} introduced the notion of minimal multiplicity for modules as follows:

\begin{chunk} \label{MM} \label{mm} Let $R$ be a ring, $I$ be an $\fm$-primary ideal of $R$, and let $M$ be a nonzero torsion-free $R$-module. Then $M$ is said to have \emph{minimal multiplicity with respect to $I$} if and only if the equality $\e_R(I,M)=\len_R(IM/I^2M)$ holds. Here,
$\e_R(I,M)$ denotes the (Hilbert-Samuel) multiplicity of $M$ with respect to $I$, and $\len_R(IM/I^2M)$ denotes the length of $IM/I^2M$ over $R$; see~\cite[Defn. 15]{TonyP}. If $M$ has minimal multiplicity with respect to $\fm$, we simply say $M$ has \emph{minimal multiplicity}. Note that, if $M=R$ and $M$ has minimal multiplicity, then the multiplicity of $R$ equals its embedding dimension, that is, $R$ is a ring of minimal multiplicity in the sense of Abhyankar~\cite{Abhyankar}.
\end{chunk}

\begin{chunk} \label{B} \label{Ulrich} Let $R$ be a ring, $I$ be an $\fm$-primary ideal of $R$, and let $M$ be a torsion-free $R$-module. Assume $(x)$ is a minimal reduction of $I$.  
\begin{enumerate}[\rm(i)]
\item It follows that $M$ is a Cohen-Macaulay $R$-module, $(x)$ is a parameter ideal for $M$, and $\e_R(I,M)=\e_R\left((x),M\right)=\len_R(M/xM)$; see, for example,~\cite[4.6.5, 4.7.9, and 4.7.11]{BH}.
\item Note that $M/IM$ has finite length and so $\e_R(I, M/IM)=0$. Thus 
\[
\e_R(I, M)=\e_R(I,IM)+\e_R(I, M/IM)=\e_R(I,IM).
\]
Moreover, by part (i), we have that $\e_R(I, IM)=\len_R(IM/xIM)$. Therefore, 
\[
\e_R(I, M)=\e_R(I,IM)=\len_R(IM/xIM).
\]
\item Assume $I^2M=xIM$. Then $\len_R(IM/xIM)=\len_R(IM/I^2M)$ and so, due to part (ii), $M$ has minimal multiplicity with respect to $I$.
\item Assume $M$ has minimal multiplicity with respect to $I$. Then, in view of Definition~\ref{MM} and part (ii), we have that $\len_R(IM/I^2M)=\e_R(I, M)=\len_R(IM/xIM)$. So, $I^2M=xIM$.
\end{enumerate}
\end{chunk}

Parts (iii) and (iv) of \ref{B} yield the following useful characterization of when a torsion-free module has minimal multiplicity with respect to an ideal:

\begin{chunk} \label{mm-defn} Let $R$ be a ring, $I$ be an $\fm$-primary ideal of $R$, and let $M$ be a torsion-free $R$-module. Then $M$ has minimal multiplicity with respect to $I$ if and only if $I^2M=xIM$ for every minimal reduction $(x)$ of $I$ if and only if $I^2M=yIM$ for some minimal reduction $(y)$ of $I$.
\end{chunk}

In the following, we discuss how our interest in studying modules of minimal multiplicity developed, including the initial motivation and key inspirations.

\begin{diss} Our motivation to study modules of minimal multiplicity comes from the notion of Arf property. Recall that $R$ is called an \emph{Arf ring} if every integrally closed $\fm$-primary ideal $I$ of $R$ is stable, that is,  $I^2=xI$ for some $x\in I$; see~\cite[2.2]{Lipman}. Our aim was to extend the Arf property from rings to modules, and for that we defined an Arf module\footnote{We initially defined Arf modules independently of~\cite{TonyP}, before its relevance to our work became apparent, and prior to the appearance of the manuscripts~\cite{Souvik2, Souvik1}.} as follows: Given a nonzero torsion-free $R$-module $M$ and an $\fm$-primary ideal $I$ of $R$, we say $M$ is \emph{$I$-Arf} (or Arf with respect to $I$) provided that there exists a minimal reduction $(x)$ of $I$ such that $I^2M=xIM$. According to our definition, the following implications hold:  $R$ is an Arf ring if and only if $R$ is an $I$-Arf $R$-module for each integrally closed $\fm$-primary ideal $I$ of $R$ if and only if each torsion-free $R$-module is an $I$-Arf $R$-module for each integrally closed $\fm$-primary ideal $I$ of $R$.

It is clear by \ref{mm-defn} that $M$ is $I$-Arf if and only if $M$ has minimal multiplicity with respect to $I$. In case the ideal $I$ in question has no minimal reductions, these two definitions do not agree. For example, if $R=\ZZ_2[\![x,y]\!]/\left(xy(x+y)\right)$ and $M=\fm^n$ for some $n\gg 0$, then $M$ has minimal multiplicity (since it is Ulrich), but $\fm$ has no reductions so Arf module property is undefined; see~\cite[2.1]{BHU},~\cite[8.3.2]{HunekeSwanson}, and Remark~\ref{02}. As we will need to assume the existence of minimal reductions simply to define Arf-modules with respect to $\fm$-primary ideals, we will be working with nothing but modules of minimal multiplicity. For this reason, we will keep using the previously defined terminology, namely modules of minimal multiplicity~\cite{TonyP}. 

When we introduced the notion of an Arf module with respect to an ideal $I$, we argued that the defining equality is independent of the choice of a minimal reduction of $I$. This can be seen from~\cite[Thm. 2.9 (4)$\Leftrightarrow$(5)]{RV} (we make use of the theorem by setting $J=(x)$, $M_1=IM$, and $M_2=I^2M$). On the other hand, it is worth noting that the equality $I^2M=xIM$ may depend on the choice of a minimal reduction of $I$ over higher dimensional rings. For example, there exists a two-dimensional Cohen-Macaulay local ring $R$ and parameter ideals $\fkq$ and $\fkq'$ of $R$ such that $\fkq$ and $\fkq'$ are reductions of $I$, $I^3=\fkq I^2$, $I^3\ne \fkq' I^2$, and $I^4=\fkq' I^3$. Thus, the equality $I^3 M = \fkq I^2 M$ depends on the choice of $\fkq$; see~\cite[3.3]{Marley}.
\end{diss}

The following fact can be seen from~\cite[4.1, 4.2, and 4.4]{Souvik2}; see also~\ref{mm} and~\ref{mm-defn}. 

\begin{chunk} \label{corchar} Let $M$ be a nonzero torsion-free $R$-module. Then the following conditions are equivalent:
\begin{enumerate}[\rm(i)]
\item $\fm^i M$ has minimal multiplicity for all $i\geq 0$.
\item $M$ has minimal multiplicity.
\item $\e_R(M)=\mu_R(\fm M)$. 
\item $\fm^2 M \cong \fm M$.
\end{enumerate}
\end{chunk}

An independent proof of~\ref{corchar} is given in~\autoref{apa}, using a certain short exact sequence established in~\ref{Aseq}. The characterization presented in~\ref{corchar} is used in Lemma~\ref{m-dual}, Example~\ref{Arfsyz}, and~\ref{Burch-facts}. We proceed by recording an example of a module of minimal multiplicity using~\ref{corchar}; several further such examples will be constructed later in the sequel. 

\begin{eg} Let $R=k[\![t^5, t^6, t^7]\!]$ and consider the ideal $M=(t^{10}, t^{11}, t^{12}, t^{13})$ of $R$. It follows that $\fm M=(t^{15}, t^{16}, t^{17}, t^{18}, t^{19})$ and hence $\mu_R(\fm M)=5$. Since $\e_R(M)=\e(R)=5$, we conclude by~\ref{corchar} that the $R$-module $M$ has minimal multiplicity.
\end{eg}

The next remark gathers some useful observations about stable ideals and minimal multiplicity for torsion-free modules, which will come up again later.

\begin{rmk} \label{SI} Let $R$ be a ring.
\begin{enumerate}[\rm(i)]
\item Let $I$ be an $\fm$-primary ideal of $R$ and let $M$ be a torsion-free $R$-module. Assume $I$ is a \emph{stable} ideal, that is, $I^2=yI$ for some $y\in I$; see~\cite[1.3]{Lipman}. Then $(y)$ is a minimal reduction of $I$ and $I^2M=yIM$. Thus, $M$ has minimal multiplicity with respect to $I$.
\item Note that $\fm$ is stable if and only if $R$ has minimal multiplicity, that is, $\e(R)=\mu_R(\fm)$. Thus, if $R$ has minimal multiplicity, then every torsion-free $R$-module has minimal multiplicity; see part~(i).
\item Assume $R$ is Arf, for example, $\e(R)\leq 2$. Then $R$ has minimal multiplicity, and thus every torsion-free $R$-module has minimal multiplicity; see~\cite[page 664]{Lipman} and part~(ii).
\item If $R$ has minimal multiplicity as an $R$-module, then it follows from the definition that $\fm$ is stable, and hence $R$ has minimal multiplicity as a ring; see part (ii). So, every torsion-free $R$-module has minimal multiplicity if and only if $R$ has minimal multiplicity (as a ring).
\end{enumerate}
\end{rmk}

The following result can be used to obtain stable ideals and modules of minimal multiplicity:

\begin{chunk} \label{prestable} Let $R$ be a ring and let $I$ be an $\fm$-primary ideal of $R$. Assume $I^{n}$ is generated by $n$-elements for some $n\geq 2$. Then $I^{n-1}$ is a stable ideal; see~\cite[Cor. 1]{PSI}.  
\end{chunk}

\begin{eg} Let $R=k[\![t^5, t^6, t^7, t^8]\!]$ and $I=(t^5, t^7)$. Then $I^3=(t^{15}, t^{17}, t^{19})$. Thus, $I^{2}$ is a stable ideal so that every torsion-free $R$-module has minimal multiplicity with respect to $I^{2}$; see Remark~\ref{SI}(i) and~\ref{prestable}. 
\end{eg}

Next, we recall the definition of Ulrich modules and record a few related facts. We study the relation between Ulrich modules and those of minimal multiplicity in Section~\ref{section-MMB} in more detail.

\begin{chunk} [{\cite[4.1]{DMS}}] \label{UM} \label{min-Ulrich} Let $R$ be a ring, $I$ be an $\fm$-primary ideal of $R$, and let $M$ be a nonzero torsion-free $R$-module. Then $M$ is called an \emph{$I$-Ulrich} provided that $\e_R(I,M)=\len_R(M/IM)$. If $M$ is $\fm$-Ulrich, we simply say $M$ is an \emph{Ulrich} $R$-module. Note that $\fm$ is an Ulrich $R$-module if and only if $e(R)=\e_R(\fm, \fm)=\len_R(\fm/\fm^2)=\mu_R(\fm)$. Therefore, $\fm$ is an Ulrich $R$-module if and only if $R$ has minimal multiplicity (as a ring).
\end{chunk}

As pointed out in~\cite[2.8]{Souvik1}, every Ulrich module has minimal multiplicity: 

\begin{rmk} \label{02} Let $R$ be a ring, $I$ be an $\fm$-primary ideal of $R$, and let $M$ be a torsion-free $R$-module. Assume $(x)\subseteq I$ is a minimal reduction of $I$. It follows from~\ref{B}(i) that:
\begin{align}
\notag{} M \textnormal{ is $I$-Ulrich } \Longleftrightarrow  \len_R(M/IM)=\e_R(I,M)=\len_R(M/xM) \Longleftrightarrow IM=xM.
\end{align}
Therefore, if $M$ is $I$-Ulrich, then $I^2M=xIM$ so that $M$ has minimal multiplicity with respect to $I$; see~\ref{B}(iv).
\end{rmk}

\begin{rmk} \label{mA} Let $R$ be a ring and let $M$ be a nonzero torsion-free $R$-module. If $(x)$ is a minimal reduction of $\fm$, then the following implications hold:
\begin{align} \label{mA1} 
\tag{\ref*{mA}.1}  
M \text{ has minimal multiplicity } 
&\Longleftrightarrow  \fm^2M = x \fm M   
\Longleftrightarrow  \fm (\fm M) = x(\fm M)   \notag\\
&\Longleftrightarrow  \fm M \text{ is Ulrich }   \notag
\Longleftrightarrow  \fm^2 M \cong \fm M.
\end{align}
Here, in \eqref{mA1}, the last implication is due to~\cite[4.6]{DMS}. Next, we use \eqref{mA1} and note:
\begin{align} \label{mA2} 
\tag{\ref*{mA}.2}  
M \text{ has minimal multiplicity }  
\Longrightarrow \fm M \text{ is Ulrich}  
&\Longrightarrow \fm^2 M \text{ is Ulrich}  \\  \notag{} 
&\;\;\;\;\;\;\;\;\;\;\;\;\;\; \vdots  \\  \notag{} 
&\Longrightarrow \fm^i M \text{ has minimal multiplicity}.
\end{align}
\end{rmk}

In what follows, we note that the $R$-module $\fm$ has minimal multiplicity (with respect to $\fm$) whenever the multiplicity of $R$ is at most three. 

\begin{rmk} \label{MU} Let $R$ be a ring and set $e=\e(R)$.
\begin{enumerate}[\rm(i)]
\item It follows that $\fm^{e-1}$ is an Ulrich $R$-module; see~\cite[2.1]{BHU}. Thus, by~\ref{Ulrich}, $\fm^{e-1}$ has minimal multiplicity. 
\item If $e=3$, then $\fm^2$ is an Ulrich $R$-module so that $\fm$ has minimal multiplicity; see part~(i) and~\eqref{mA2}. For example, if $R=k[\![t^3, t^5]\!]$, then $e=3$ and $\fm$ has minimal multiplicity, but $R$ does not have minimal multiplicity.
\item If $e\leq 3$, then $\fm$ has minimal multiplicity; see parts~(i) and~(ii).
\end{enumerate}
\end{rmk}  

In contrast to Remark~\ref{MU}, there exist rings of multiplicity four whose maximal ideal does not have minimal multiplicity. For example, if $R = k[\![t^4, t^5, t^{11}]\!]$, then 
$\fm^2 = (t^8, t^9, t^{10})$ and $\e(R) = 4 = \e_R(\fm) \neq 3 = \mu_R(\fm^2)$ so that $R$-module $\fm$ does not have minimal multiplicity. Here, $R$ is not Gorenstein; see \ref{Burch-facts}(iv) for a similar example involving a Gorenstein ring.

The converse of the fact stated in \eqref{mA2} may fail; that is, even if a torsion-free $R$-module $M$ does not have minimal multiplicity, $\fm M$ still can have minimal multiplicity.

\begin{eg} \label{Arfmax} \label{eg1} Let $R=k[\![t^4, t^5, t^6]\!]$. 
\begin{enumerate}[\rm(i)]
\item Set $\fq=(t^4)$. Then $\fm=\fq+(t^5, t^6)$ so that $\fm^2$ is equal to
\begin{align} \notag{}
\fq\fm+(t^5, t^6)\fq+(t^5, t^6)^2
&=\fq\fm+(t^5, t^6)^2   
=(t^8, t^9, t^{10})+(t^{10}, t^{11}, t^{12})=(t^8, t^9, t^{10}, t^{11}).   \notag{}
\end{align}
Also, $\fm^2=\fq\fm+(t^{11})$ so that $\fm^3=\fq\fm^2+(t^{11})\fm$. Note that $(t^{11})\fm=(t^{15},t^{16},t^{17})\subseteq \fq\fm^2$.
Thus, $\fm^3=\fq\fm^2$, that is, $\fq$ is a minimal reduction of $\fm$. So, $\fm^{n+2}=\fq\fm^{n+1}$ for all $n\geq 1$. Thus, if $M$ is a torsion-free $R$-module, then $\fm^2 (\fm^n M)=\fq \fm (\fm^n M)$, and hence $\fm^n M$ has minimal multiplicity for all $n\geq 1$; see (\ref{mA}.1). In particular, letting $M=R$ and $n=1$, we conclude that $\fm$ has minimal multiplicity.
\smallskip

\item Let $M=R+Rt^7$. Then $M$ is not a free $R$-module. Note that $\fm \subseteq \fm M \subseteq M$ since $1\in M$. Moreover, we see $\fm M=(t^4, t^5, t^6)(1, t^7)\subseteq \fm$ so that $\fm M=\fm$. Thus, $\fm M$ has minimal multiplicity by part (i). On the other hand, if $M$ has minimal multiplicity, then (\ref{mA}.2) implies that $\fm M=\fm$ is Ulrich, and hence $R$ has minimal multiplicity. As $R$ does not have minimal multiplicity, we conclude that $M$ does not have minimal multiplicity.
\end{enumerate}
\end{eg}

\begin{rmk} Recall that, if $R$ is a ring and an $\fm$-primary ideal $I$ of $R$ is stable, then every torsion-free $R$-module has minimal multiplicity with respect to $I$; see~\ref{SI}(i). The converse of this fact is not necessarily true. In Example~\ref{Arfmax}, $\fm$ is not a stable ideal, but the $R$-module $\fm$ has minimal multiplicity; see~\ref{SI}(ii).
\end{rmk}

It is noted in \ref{02} that every Ulrich module has minimal multiplicity. On the other hand, as we point out by the next example, there exist nonzero modules of minimal multiplicity that are not Ulrich.

\begin{eg} \label{minmulnotulrich} Let $R=k[\![t^4, t^5, t^6]\!]$. We know by Example~\ref{Arfmax} that $\fm$ has minimal multiplicity (with respect to $\fm$). On the other hand, since $R$ does not have minimal multiplicity, $\fm$ is not an Ulrich $R$-module; see Remark~\ref{min-Ulrich}. 
\end{eg}

We can construct various examples of modules as in Example~\ref{minmulnotulrich}; our next result shows how to do this.

\begin{prop} \label{reduction} Let $R$ be a ring, $I$ be a non-principal $\fm$-primary ideal of $R$, and let $(x)$ be a minimal reduction of $I$.
Set $r=\min\{n\geq 0: I^{n+1}=x  I^{n}\}$. If $M$ is a torsion-free $R$-module, then $I^{r-1}M$ has minimal multiplicity with respect to $I$. Moreover, $I^{r-1}$ is not an $I$-Ulrich $R$-module. 
\end{prop}

\begin{proof} As $I$ is not principal, we have that $r\geq 1$. So $I^{r+1}=x I^{r}$, that is, $I^2  I^{r-1}=x I I^{r-1}$, and so $I^{r-1}M$ is an $I$-Arf module. On the other hand, if $I^{r-1}$ is an $I$-Ulrich $R$-module, then \ref{Ulrich} implies that $I^r=II^{r-1}=xI^{r-1}$. However, this equality contradicts the minimality of $r$. Therefore, $I^{r-1}$ is not an $I$-Ulrich module.
\end{proof}

In general, a torsion-free module $M$ over a ring $R$ may have minimal multiplicity (with respect to $\fm$), but not necessarily with respect to other $\fm$-primary ideals. We give such an example next:

\begin{eg} \label{hepArf} Let $R=k[\![t^3, t^7, t^{11}]\!]$. 
\begin{enumerate}[\rm(i)]
\item Let $I=\overline{(t^6)}$ be the integral closure of the ideal of $R$ generated by $t^6$. It follows from~\cite[1.5.2 and 1.6.1]{HunekeSwanson} that 
\[
I=\overline{(t^6)}=\overline{(t^6)\overline{R}} \cap R=(t^6)\overline{R} \cap R=t^6k[\![t]\!]\cap R=(t^6, t^7, t^{11}).
\]
Moreover, 
\[
I^2=(t^{12},t^{13}, t^{14}, t^{17}, t^{18}, t^{22})\neq (t^{12},t^{13}, t^{17})=t^6I.
\]
This implies that $I$ is not a stable ideal of $R$.
\item If $J$ is an $\fm$-primary ideal of $R$, then $R$ has minimal multiplicity with respect to $J$ if and only if $J$ is a stable ideal. So, if $I$ is as in part (i), then $R$ does not have minimal multiplicity with respect to $I$. On the other hand, $R$ has minimal multiplicity with respect to $\fm$ because the ring $R$ has minimal multiplicity.
\end{enumerate}
\end{eg}

\begin{rmk} \label{conductor} Let $R$ be a generically Gorenstein non-Gorenstein ring with a canonical ideal $\omega_R$ and total ring of fractions $\rm Q$. Pick a minimal reduction $(x)$ of $\omega_R$ 
and set $K=x^{-1}\omega_R\subseteq Q$. Then $K$ is the fractional canonical ideal of $R$. Now, consider the \emph{conductor} $\con$ of $R$ in $R[K]$, namely $\con=(R:_{\rm Q}R[K])$. Note that $\con=\con R\subseteq \con K$ since $R\subseteq K$. Also, $\con K \subseteq \left(R:_QR[K]\right)=\con$. Thus, $\con K =\con$, and hence $\con \cdot \omega_R = \con \cdot x$. So, $\con$ is $\omega_R$-Ulrich. Therefore, $\con$ has minimal multiplicity with respect to $\omega_R$; see Remark~\ref{02}. 

In fact, this argument remains valid not only for $\omega_R$, but for all $\fm$-primary ideals $I$ of $R$. More precisely, if $I$ is such an ideal of $R$, $(x)$ is a minimal reduction of $I$, and $K=x^{-1} I \subseteq Q$, then $\left(R:_QR[K]\right)$ is $I$-Ulrich, and thus has minimal multiplicity with respect to $I$.
\end{rmk}

We use Remark \ref{conductor} and compute new examples of modules of minimal multiplicity.

\begin{eg} \label{e1} $\phantom{}$
\begin{enumerate}[\rm(i)]
\item Let $R=k[\![t^3, t^7, t^{8}]\!]$. Note that $K=R+Rt \cong \omega_R$ and so $R[K]=R[t]=k[\![t]\!]=\overline{R}$, the integral closure of $R$ in $Q$. Then $\con=(R:_Q\overline{R})=t^6\overline{R}=(t^6, t^7, t^8)R$. Therefore, by Remark \ref{conductor}, the ideal $(t^6, t^7, t^8)R$ is $\omega_R$-Ulrich, and hence has minimal multiplicity with respect to $\omega_R$.
\item Let $R=k[\![t^4, t^9, t^{11}, t^{14}]\!]$. Then $$K=R+Rt^3+Rt^5\cong \omega_R \text{ and } R[K]=R+Rt^3+Rt^5+Rt^6.$$ 
It follows that $(R:_RRt^3)=(t^8, t^9, t^{11}, t^{14})R$ and $ (R:_RRt^5)= (R:_RRt^6)= (t^4, t^9, t^{11}, t^{14})R$. This implies that:
\begin{align}
\notag  \con=(R:_QR[K])& =(R:_RRt^3) \cap (R:_RRt^5) \cap (R:_RRt^6) \\ & \notag{} = (t^8, t^9, t^{11}, t^{14})R \cap (t^4, t^9, t^{11}, t^{14})R\\ & = \notag{} (t^8, t^9, t^{11}, t^{14})R
\end{align}
Therefore, by Remark \ref{conductor}, the ideal $(t^8, t^9, t^{11}, t^{14})R$ is $\omega_R$-Ulrich, and has minimal multiplicity with respect to $\omega_R$. 
\end{enumerate}
It seems worth noting that, $R[K]=\overline{R}$ for the ring $R$ in part (i); on the other hand, for the ring $R$ in part (ii), we have that $R[K]\neq \overline{R}$ and $\con\neq(R:_Q\overline{R})=(t^{11}, t^{12}, t^{13}, t^{14})R$.
\end{eg}

\subsection*{Minimal multiplicity over idealizations}

In this subsection, we provide a characterization of the minimal multiplicity of
idealizations \(A = R \ltimes M\) in terms of the minimal multiplicity of the ring
\(R\) and the \(R\)-module \(M\). Furthermore, we point out that an idealization
may fail to have minimal multiplicity, even when \(M\) does as an \(R\)-module.

We recall the relevant definition: Given an \( R \)-module \( M \), we denote by \( R \ltimes M \) the \emph{idealization} of \( M \) over \( R \). As an abelian group, \( R \ltimes M \) is just \( R \oplus M \), but the multiplication is defined by $(a, x)\cdot (b, y) = (ab, ay + bx)$ for all \( (a, x), (b, y) \in R \ltimes M \). Note that $0 \ltimes M$ is an ideal of $R \ltimes M$, $(0 \ltimes M)^2=0$, and $R \cong (R \ltimes M)/(0 \ltimes M)$. As we assume $R$ is a one-dimensional Cohen-Macaulay local ring, if $M$ is a torsion-free $R$-module, then $R \ltimes M$ is a one-dimensional Cohen-Macaulay local ring with maximal ideal $\fm \ltimes M$ and $(R \ltimes M)/(\fm \ltimes M) \cong R/\fm$; see, for example,~\cite[Page 2]{Nagata}.

The following properties follow from the definitions:

\begin{chunk} \label{lem-min-mul} Let $R$ be a ring, $M$ be a torsion-free $R$-module, and let $A=R \ltimes M$ be the idealization of $M$ over $R$. Then the following hold:
\begin{enumerate}[\rm(i)]
\item $\edim(A)=\edim(R)+\mu_R(M)$.
\item $\e(A)=\e(R)+\e_R(M)$.
\end{enumerate}
\end{chunk}

\begin{prop} \label{prop-min-mul} Let $R$ be a ring, $M$ be a torsion-free $R$-module, and let $A=R \ltimes M$ be the idealization of $M$ over $R$. Then the following conditions are equivalent:
\begin{enumerate}[\rm(i)]
\item The ring $A$ has minimal multiplicity.
\item Each torsion-free $A$-module has minimal multiplicity.
\item The ring $R$ has minimal multiplicity and $M$ is an Ulrich $R$-module.
\end{enumerate}
\end{prop}

\begin{proof} The equivalence of parts (i) and (ii) follow by Remark~\ref{SI}(ii).

To show part (i) implies part (iii), we assume that $A$ has minimal multiplicity (as a ring). Then, by~\ref{lem-min-mul}, we have:
\[
\edim(R)+\mu_R(M)=\edim(A)=\e(A)=\e(R)+\e_R(M).
\]
This yields:
\[
0\leq \e(R)-\edim(R)= \mu_R(M)-\e_R(M)\leq 0.
\]
Therefore, $\edim(R)=\e(R)$ and $ \mu_R(M)=\e_R(M)$, that is, the ring $R$ has minimal multiplicity and the $R$-module $M$ is Ulrich.

Next, assume part (iii) holds. Then~\ref{lem-min-mul} implies that
\[
\e(A)=\e(R)+\e_R(M)=\edim(R)+\mu_R(M)=\edim(A).
\]
Therefore, the ring $A$ has minimal multiplicity.
\end{proof}

In the next remark, we use Proposition~\ref{prop-min-mul} to construct 
idealizations with minimal multiplicity and to emphasize that an idealization 
may fail to have minimal multiplicity, even when the module \(M\) does as 
an \(R\)-module.

\begin{rmk} \label{cor-min-mul} \phantom{} Let $R$ be a ring and set $A=R \ltimes R$. Assume $R$ has minimal multiplicity. Then, by Proposition~\ref{prop-min-mul}, the ring $A$ has minimal multiplicity if and only if $R$ is an Ulrich $R$-module, or equivalently, $R$ is regular. Thus, if $R$ is not regular, then $A$ does not have minimal multiplicity as a ring, but it does have minimal multiplicity as an $R$-module; see Remark~\ref{SI}(ii).
\end{rmk}

Recall that $\fm$ is an Ulrich $R$-module if and only if $R$ has minimal multiplicity; see Remark~\ref{min-Ulrich}. So, Proposition~\ref{prop-min-mul} shows that, the ring $R$ has minimal multiplicity if and only if the ring $R \ltimes \fm$ has minimal multiplicity. More generally, if $R$ has minimal multiplicity, then $\Omega_R^i(k)$ is an Ulrich $R$-module for all $i\geq 1$; see~\cite[2.5]{BHU}. Thus, we conclude by Proposition~\ref{prop-min-mul} that: 

\begin{chunk} \label{part-of-intro} Let $R$ be a ring. Then $R$ has minimal multiplicity if and only if the idealization ring $R \ltimes \Omega_R^i(k)$ has minimal multiplicity for all $i\geq 1$.
\end{chunk}

\subsection*{Syzygies of modules of minimal multiplicity} In this subsection, we show, by means of an example, that the (co)syzygy of a module of minimal multiplicity need not have minimal multiplicity in general. We observe this phenomenon in Example \ref{Arfsyz} after some preparation. Recall that, given an $R$-module $M$, we set $M^{\ast}=\Hom_R(M,R)$.

\begin{rmk} \label{2gen-rmk} Let $R$ be a ring with total ring of fractions $\rm Q$, and let $I$ be a fractional ideal of $R$, that is, $I$ is a finitely generated $R$-submodule of $\rmQ$ and $I$ contains an element of $R$ which is a non zero-divisor on $R$. If $\mu_R(I)=2$, then $I^{\ast} \cong \Omega_R(I)$. This fact follows, for example, from~\cite[3.1]{CHSIRW} for the case where $I$ is an ideal of $R$. To establish the general case, pick an element $s\in R$ such that $s$ is a non zero-divisor on $R$ and $sI \subseteq R$. Set $J=sI$. Then $J$ is an ideal of $R$ such that $\mu_R(J)=2$. Moreover, since $I$ contains a non zero-divisor on $R$, and $s$ is a non zero-divisor on $R$, we conclude that $J$ contains a non zero-divisor on $R$. Therefore, there is a short exact sequence of $R$-modules $0 \to J^{\ast} \to R^{\oplus 2} \to J \to 0$; see, for example,~\cite[the proof of 3.1]{CHSIRW}. Since $I \cong J$ as an $R$-module, we obtain a short exact sequence of $R$-modules $0 \to I^{\ast} \to R^{\oplus 2} \to I \to 0$. This implies that $I^{\ast} \cong \Omega_R(I)$.
\end{rmk}

\begin{rmk} \label{2gen} Let $R$ be a ring and set $M=\fm^{\ast}$.
\begin{enumerate}[\rm(i)]
\item As $\fm$ contains a non zero-divisor on $R$, we have that $M = \Hom_R(\fm,R) \cong (R:_{\Q} \fm)$ as $R$-modules, where $(R:_{\Q} \fm)$ is a fractional ideal of $R$; see ~\cite[2.4.2]{HunekeSwanson}. 
\item If $R$ is not regular, then $(R:_{\Q} \fm) = (\fm:_Q\fm)\cong \Hom_R(\fm, \fm)$.
\item If $R$ is not regular, it follows from parts (i) and (ii) that
\[
M=\Hom_R(\fm,R) \cong  (\fm:_Q\fm) \Longrightarrow \fm M=\fm\Hom_R(\fm,R) \cong  \fm (\fm:_Q\fm)=\fm.
\]
\end{enumerate}
\end{rmk}

\begin{lem} \label{m-dual} Let $R$ be a ring. The following conditions are equivalent:
\begin{enumerate}[\rm(i)]
\item $R$ has minimal multiplicity.
\item Every torsion-free $R$-module has minimal multiplicity.
\item $\fm^{\ast}$ has minimal multiplicity.
\end{enumerate}
\end{lem}

\begin{proof} It is enough to assume $\fm^{\ast}$ has minimal multiplicity and see that $R$ has minimal multiplicity; see Remark~\ref{SI}(ii). Set $M=\fm^{\ast}$. The rank of $M$ equals the rank of $\fm$, which is one. Thus, $\e_R(M)=\e(R)$; see~\cite[4.7.9]{BH}. So, it follows from~\ref{corchar} and Remark~\ref{2gen}(iii) that the $R$-module $M$ has minimal multiplicity if and only if $\e(R)=\e_R(M)=\mu_R(\fm M)=\mu_R(\fm)$ if and only if the ring $R$ has minimal multiplicity.
\end{proof}

\begin{eg} \label{Arfsyz} Let $R$ be a ring such that $\e(R)=3$ and $\edim(R)=\mu_R(\fm)=2$. For example, we can pick $R=k[\![t^3,t^4]\!]$. Note that $\fm$ has minimal multiplicity; see~\ref{MU}(iii). Set $M=\fm^{\ast}$. Since $\mu_R(\fm)=2$, there is a short exact sequence of $R$-modules $0 \to M \to R^{\oplus 2} \to \fm \to 0$; see Remark~\ref{2gen-rmk}. This exact sequence implies that $\e_R(M)=2\e(R)-\e_R(\fm)=3$. Since $R$ is not regular, we have that $\fm M \cong \fm$; see Remark~\ref{2gen}(iii). Thus, $\mu_R(\fm M)=2$. This implies that $M$ does not have minimal multiplicity; otherwise,~\ref{corchar} would imply that $2=\mu_R(\fm M)=\e_R(M)=3$.

We know that $\fm^{\ast\ast}\cong \fm$; see \ref{maxref}. Assume $R$ is a (one-dimensional) Gorenstein domain. Then a result of Bass~\cite[6.3]{BassUbi} shows that $\mu_R(M)\leq 2$. If $M$ is cyclic, then $\fm$ is also cyclic (since $\fm \cong \fm^{\ast\ast}=M^{\ast}$), and hence $R$ is regular. Thus, $M$ is not cyclic, that is, $\mu_R(M)=2$. Recall that $R \subseteq (R:_{\Q} \fm) \cong M$, see Remark \ref{2gen}(i). Since $(R:_{\Q} \fm)$ is a fractional ideal of $R$ containing a non zero-divisor on $R$, there is a short exact sequence of $R$-modules $0 \to \fm \to R^{\oplus 2} \to M \to 0$; see Remark~\ref{2gen-rmk}. Therefore, even though the $R$-module $\fm$ has minimal multiplicity, its co-syzygy module $M$ does not have minimal multiplicity.
\end{eg}

\subsection*{Modules of minimal multiplicity versus Burch modules} We have already noted that each Ulrich module has minimal multiplicity, and the converse of this fact is not necessarily true; see Proposition \ref{reduction}. There is another class of modules, a subclass of Ulrich modules, which we want to touch upon, namely, Burch modules. A torsion-free $R$-module $M$ is called \emph{Burch} if there is an $x\in \fm$ such that $x$ is a non zero-divisor on $M$ and $k$ is a direct summand of $M/xM$; see~\cite{SDTK} for examples and details.

In the following, we point out that the class of modules of minimal multiplicity and that of Burch modules (over one-dimensional Cohen-Macaulay rings) are distinct, in general. We also remark that Burch modules of minimal multiplicity need not be Ulrich.

\begin{chunk} \label{Burch-facts} Let $R$ be a ring.
\begin{enumerate}[\rm(i)]
\item The $R$-module $\fm$ is always a Burch module; see~\cite[3.4]{SDTK}.
\item If $M$ is a torsion-free $R$-module such that $M$ is Burch and $\id_R(M)<\infty$, then $R$ is regular; see~\cite[3.19]{SDTK}.
\item Assume $R$ is not Gorenstein, has minimal multiplicity, and admits a canonical module $\omega_R$. Then $\omega_R$ also has minimal multiplicity since every torsion-free $R$-module has minimal multiplicity; see Remark~\ref{SI}(ii). On the other hand, $\omega_R$ is not a Burch $R$-module; see part~(ii). For example, if $R = k[\![t^3, t^4, t^5]\!]$, the canonical module $\omega_R$ fails to be Burch, yet it has minimal multiplicity as an $R$-module.
\item Let $R=k[\![t^4, t^5]\!]$ and set $M=\fm$. Then $M$ is a Burch $R$-module, but it does not have minimal multiplicity since $\e(R)=4\neq 3=\mu_R(\fm^2)$; see~part (i) and~\ref{corchar}. 
\item If $R=k[\![t^4, t^5, t^6]\!]$ and $M=\fm$, then $M$ is Burch and has minimal multiplicity, but it is not Ulrich; see Example~\ref{Arfmax}.
\end{enumerate}
\end{chunk}

\section{Modules of minimal multiplicity versus Ulrich modules} \label{section-MMB}

Recall that, unless otherwise stated, $R$ denotes a \emph{one-dimensional commutative Noetherian Cohen-Macaulay local ring} with maximal ideal $\fm$ and residue field $k$. Moreover, all $R$-modules are assumed to be finitely generated.

A torsion-free $R$-module with minimal multiplicity with respect to an $\mathfrak{m}$-primary ideal $I$ of $R$ is not necessarily $I$-Ulrich; see Example~\ref{minmulnotulrich} and Proposition~\ref{reduction}. In this section, we prove Theorem~\ref{AGchar} and identify certain classes of modules and ideals for which the Ulrich property and minimal multiplicity are equivalent; see also Corollary~\ref{AGA}. Before proving Theorem~\ref{AGchar}, we establish several preliminary results. We begin by setting $\overline{(-)} = - \otimes_R R/xR$ for a given ring $R$ and an element $x \in R$.

\begin{lem} \label{ArfUlrich} Let $R$ be a ring, $I$ be an $\fm$-primary ideal of $R$, $(x)\subseteq I$ be a minimal reduction of $I$, and let $M$ be a torsion-free $R$-module. Then the following conditions are equivalent:
\begin{enumerate}[\rm(i)]
\item $M$ has minimal multiplicity with respect to $I$.
\item $IM$ is $I$-Ulrich.
\item $I \subseteq \Ann_R\left(\overline{IM}\right)$.
\end{enumerate}
\end{lem}

\begin{proof} We have by \ref{mm-defn} that $I^2M=xIM$ if and only if $M$ has minimal multiplicity with respect to $I$. Also, Remark \ref{02} shows that $I \cdot IM = x \cdot IM$ if and only if $IM$ is $I$-Ulrich. So, the equivalences hold.
\end{proof}

We leave the proof of the following observation to the reader.

\begin{lem} \label{obs} Let $R$ be a ring and let $M$ and $N$ be torsion-free $R$-modules. If $x\in \fm$ is a non zero-divisor on $R$, then there is an injective $R$-module homomorphism as follows:
\[
\overline{\Hom_R(M,N)} \xhookrightarrow{} \Hom_{\overline{R}}\left(\overline{M}, \overline{N}\right).
\]
\end{lem}

The next proposition plays an important role in the proof of Theorem~\ref{AGchar}.

\begin{prop} \label{LAGAprop} Let $R$ be a ring, $I$ be an $\fm$-primary ideal of $R$, and let $M$ be a torsion-free $R$-module. Assume $M$ has minimal multiplicity with respect to $I$. Then $\Hom_R(IM, N)$ is $I$-Ulrich for each torsion-free $R$-module $N$.
\end{prop}

\begin{proof} Let $N$ be a torsion-free $R$-module. Then $\Hom_R(IM, N)$ is a torsion-free $R$-module. Moreover, $IM$ is torsion-free since it is a submodule of the torsion-free $R$-module $M$. Pick a minimal reduction $(x)\subseteq I$; see \ref{setup}. Then, since $x$ is a non zero-divisor on $R$, there is an injection $\overline{\Hom_R(IM,N)} \hookrightarrow \Hom_{\overline{R}}\left(\overline{IM}, \overline{N}\right)$ due to Lemma \ref{obs}. Since $M$ has minimal multiplicity with respect to $I$, Lemma \ref{ArfUlrich} shows that $I \subseteq \Ann_R\left(\overline{IM}\right)$. Therefore, we have $$I \subseteq \Ann_R\left(\overline{IM}\right)\subseteq \Ann_R\left(\Hom_{\overline{R}}(\overline{IM}, \overline{N})\right).$$ This implies that $I \cdot \overline{\Hom_R(IM,N)}=0$, that is, $I \cdot \Hom_R(IM,N)= x \cdot \Hom_R(IM,N)$. Consequently, by Remark \ref{02}, $\Hom_R(IM, N)$ is $I$-Ulrich.
\end{proof}

An application of Proposition~\ref{LAGAprop} with $I=\fm$ and $N=R$ yields:

\begin{cor} Let $R$ be a ring and let $M$ be a nonzero torsion-free $R$-module. If $M$ has minimal multiplicity, then $(\fm M)^{\ast}$ is an Ulrich $R$-module.
\end{cor}

We recall some properties which will be used, for example, in the proofs of Theorems~\ref{AGchar} and~\ref{corchar1}.

\begin{chunk} \label{fractional}  \label{trace} Let $R$ be a ring with total ring of fractions $\rm Q$, and let $I$ be an $\fm$-primary ideal of $R$. Then the following hold:
\begin{enumerate}[\rm(i)]
\item $I\left(I:_{\rmQ}I\right)= I$. 
\item $I$ is reflexive, that is, $I\cong I^{\ast\ast}$ as $R$-modules, if and only if  $\left(R:_{\rmQ}(R:_{\rmQ}I)\right)=I$; see, for example,~\cite[2.4(4)]{KoTa}.
\item Recall that $I$ is called stable if and only if $I^2=zI$ for some $z\in I$. If $I=x\left(I:_{\Q} I\right)$ for some $x\in I$, then $(x^{-1}I)^2=x^{-1}I$ so that $I^2=xI$, that is, $I$ is stable; see \cite[1.11]{Lipman}. 
\item $I$ is called a \emph{trace ideal} if it is equal to the trace of an $R$-module. It follows that $I$ is a trace ideal of $R$ if and only if $\left(I:_{\rmQ} I\right)=\left(R:_{\rmQ}I\right)$; see \cite[2.2]{SRK}. 
For example, if $R$ is not regular, then $\fm$ is a trace ideal of $R$; see~\cite[3.15]{AGR}. Moreover, each Ulrich ideal is a trace ideal; see~\cite[Page 4]{EndoGoto}.  
\end{enumerate}
\end{chunk}

\begin{chunk} \label{ideal} Let $R$ be a ring with total ring of fractions $\rm Q$.
\begin{enumerate}[\rm(i)]
\item If $I$, $J$, and $K$ are fractional ideals of $R$, then $\left(I:_{\rm Q}JK\right)=\left(\left(I:_{\rm Q}J\right):_{\rmQ}K\right)$.
\item If $I$ is an ideal of $R$ containing an element which is a non zero-divisor on $R$, and $J$ is an $R$-submodule of $\rm Q$, then $\left(J:_{\rmQ}I\right) \cong \Hom_R(I,J)$; see, for example,~\cite[2.4.2]{HunekeSwanson}.
\end{enumerate}
\end{chunk}

The observation we give next is key for the proofs of Theorems~\ref{AGchar} and~\ref{corchar1}.

\begin{chunk} \label{AGchar0} Let $R$ be a non-Gorenstein ring with a canonical ideal $\omega_R$ and let $I$ be an $\fm$-primary ideal of $R$. The following equality and isomorphisms hold:
\begin{align} \tag{\ref{AGchar0}.1}
\Hom_R(I \cdot \omega_R , \omega_R) & \cong (\omega_R:_{\rm Q} I \cdot \omega_R)  =  \left((\omega_R:_{\rm Q} \omega_R):_{\rm Q} I\right) \cong \left(R:_{\rm Q}I\right) \cong I^{\ast}.
\end{align}
Here, the isomorphisms follow from \ref{ideal}(ii) since both $I$ and $\omega_R$ are $\fm$-primary ideals of $R$, and therefore contain non zero-divisors on $R$. Moreover, the equality is due to \ref{ideal}(i).
\end{chunk}

Next is the main result of this section: 

\begin{thm} \label{AGchar} Let $R$ be a generically Gorenstein non-Gorenstein ring with a canonical ideal $\omega_R$ and let $I$ be an $\fm$-primary ideal of $R$. 
\begin{enumerate}[\rm(i)] 
\item If the $R$-module $I$ has minimal multiplicity with respect to $\omega_R$, then $I^{\ast}$ is $\omega_R$-Ulrich.
\item Assume $I$ is a trace ideal of $R$. If $I^{\ast}$ is $\omega_R$-Ulrich, then $I$ is $\omega_R$-Ulrich.
\end{enumerate}
\end{thm}

\begin{proof} For part (i), assume $I$ has minimal multiplicity with respect to $\omega_R$. It follows from Proposition~\ref{LAGAprop} that $\Hom_R(I \cdot \omega_R, \omega_R)$ is $\omega_R$-Ulrich. Thus, (\ref{AGchar0}.1) shows that $I^{\ast}$ is $\omega_R$-Ulrich and establishes part (i). 

For part (ii), assume $I$ is a trace ideal of $R$ and $I^{\ast}$ is $\omega_R$-Ulrich. Then, since $\left(R:_{\rm Q}I\right) \cong I^{\ast}$, it follows that $(R:_{\rm Q}I) \cdot \omega_R=(R:_{\rm Q}I) \cdot x$ for some minimal reduction $(x) \subseteq \omega_R$; see~(\ref{AGchar0}.1) and Remark~\ref{02}. This yields the equality $I \cdot \omega_R=I \cdot x$ since
\[
I \cdot (R:_{\rm Q}I) \cdot \omega_R=I \cdot (R:_{\rm Q}I) \cdot x \Longrightarrow I \cdot \left(I:_{\rmQ} I\right)\cdot \omega_R=I \cdot \left(I:_{\rmQ} I\right)\cdot x \Longrightarrow I \cdot \omega_R=I \cdot x.
\]
Here, the first and second implications hold by~\ref{fractional}(iv) and~\ref{fractional}(i), respectively. Therefore, $I$ is $\omega_R$-Ulrich; see~Remark \ref{02}. 
\end{proof}

If $I$ is $\omega_R$-Ulrich, then, by Remark~\ref{02}, $I$ has minimal multiplicity with respect to $\omega_R$. When $I$ is a trace ideal, Theorem~\ref{AGchar} immediately gives the converse of this fact.

\begin{cor} \label{AGchar-cor} Let $R$ be a generically Gorenstein non-Gorenstein ring with a canonical ideal $\omega_R$ and let $I$ be an $\fm$-primary trace ideal of $R$. Then $I$ has minimal multiplicity with respect to $\omega_R$ if and only if $I$ is $\omega_R$-Ulrich.
\end{cor}

We proceed and give further consequences of Theorem~\ref{AGchar}. As in the theorem, assume $R$ is a generically Gorenstein non-Gorenstein ring with total ring of fractions $\rm Q$ and a canonical ideal $\omega_R$. Recall that the conductor of $R$, namely $\con=\left(R:_QR[K]\right)$, has minimal multiplicity with respect to $\omega_R$; see Remark~\ref{conductor}. Thus, Theorem~\ref{AGchar}(i) shows that $\con^{\ast}$ is $\omega_R$-Ulrich. As $\con^{\ast} \cong (R:_Q\con)\cong R[K]$, we obtain by Remark~\ref{02} that:

\begin{cor} \label{cor-con} Let $R$ be a generically Gorenstein non-Gorenstein ring with a canonical ideal $\omega_R$. Then the $R$-module $R[K]$ has minimal multiplicity with respect to $\omega_R$.
\end{cor}

In the following, we point out that an $\fm$-primary ideal $I$ -- which is reflexive as a module -- over a generically Gorenstein non-Gorenstein ring has minimal multiplicity with respect to $\omega_R$ if and only if its algebraic dual $I^{\ast}$ does.

\begin{cor}\label{AGA} \label{AGA-cor}
Let $R$ be a generically Gorenstein non-Gorenstein ring with canonical ideal $\omega_R$, and let $I$ be an $R$-ideal. Consider the following conditions:
\begin{enumerate}[(i)]
\item $I$ is $\omega_R$-Ulrich.
\item $I$ has minimal multiplicity with respect to $\omega_R$.
\item $I^{\ast}$ is $\omega_R$-Ulrich.
\item $I^{\ast}$ has minimal multiplicity with respect to $\omega_R$.
\end{enumerate}
Then $(i)\Rightarrow (ii)\Rightarrow (iii)\Rightarrow (iv)$. Furthermore, if $I$ is a reflexive $R$-module, then
\[
(i)\Longleftrightarrow (ii)\Longleftrightarrow (iii)\Longleftrightarrow (iv).
\]
\end{cor}

\begin{proof} The implications (i)$\Longrightarrow$ (ii) and (iii)$\Longrightarrow$ (iv) hold due to Remark~\ref{02}, and Theorem~\ref{AGchar}(i) yields (ii)$\Longrightarrow$ (iii).

Now assume $I$ is a reflexive $R$-module. It is enough to establish the implication (iv)$\Longrightarrow$ (i). For that, assume $I^{\ast}$ has minimal multiplicity with respect to $\omega_R$. Set $J=I^{\ast}=\Hom_R(I,R)$. Recall that $J \cong \left(R:_{\rm Q}I\right)$  is a fractional ideal of $R$; see~\ref{ideal}(ii). So, Theorem \ref{AGchar}(i) implies that $J^{\ast}$ is $\omega_R$-Ulrich. Since $I$ is reflexive, we have that $J^{\ast} = I^{\ast\ast}\cong I$. This completes the proof.
\end{proof}

The following fact is used in the proofs of Corollaries~\ref{AGA2} and~\ref{corchar2}.

\begin{chunk} \label{maxref} If $R$ is a (one-dimensional Cohen-Macaulay local) ring, then the maximal ideal $\fm$ of $R$ is a reflexive $R$-module; see, for example~\cite[4.1]{Faber}.
\end{chunk}

Corollary~\ref{AGA-cor}, in view of \ref{maxref}, immediately yields:

\begin{cor} \label{AGA2} Let $R$ be a generically Gorenstein non-Gorenstein ring with a canonical ideal $\omega_R$. Then the following conditions are equivalent: 
\begin{enumerate}[\rm(i)] 
\item $\fm$ is $\omega_R$-Ulrich.
\item $\fm$ has minimal multiplicity with respect to $\omega_R$.
\item $\fm^{\ast}$ is $\omega_R$-Ulrich.
\item $\fm^{\ast}$ has minimal multiplicity with respect to $\omega_R$.
\end{enumerate}
\end{cor}

Corollary~\ref{AGA2} provides a characterization of almost Gorenstein rings. Although we do not define almost Gorenstein rings or discuss their properties in detail, we refer the reader to~\cite{AGR} for the definition and further information regarding this class of rings. In particular, a generically Gorenstein, non-regular ring with a canonical ideal $\omega_R$ is almost Gorenstein if and only if its maximal ideal $\fm$ is $\omega_R$-Ulrich; see~\cite[3.11]{AGR}. This fact and Corollary~\ref{AGA2} imply:

\begin{cor} \label{AlmGor-cor} Let $R$ be a generically Gorenstein non-regular ring with a canonical ideal $\omega_R$. Then $R$ is almost Gorenstein if and only if $\fm$ has minimal multiplicity with respect to $\omega_R$ if and only if $\fm^{\ast}$ has minimal multiplicity with respect to $\omega_R$.
\end{cor}

\section{Proof of Theorem \ref{thm-intro}} \label{main-section}

The aim of this section is to study cases in which the canonical module has minimal multiplicity with respect to a trace or reflexive ideal. Our main result, Theorem~\ref{corchar1}, requires some preliminary work and leads to several consequences, including part of the results announced in the introduction; see Corollaries~\ref{corcharI} and~\ref{corchar2}.

Recall that, if $R$ is a ring, an $\fm$-primary ideal $I$ is called \emph{reflexive} if $\left(R:_{\rm Q}(R:_{\rm Q}I)\right)=I$, and called \emph{trace} if $\left(I:_{\rm Q}I\right)=\left(R:_{\rm Q}I\right)$, where $\rm Q$ is the total ring of fractions of $R$; see~\ref{fractional}.

\begin{thm} \label{corchar1} Let $R$ be a generically Gorenstein non-Gorenstein ring with canonical ideal $\omega_R$, and let $I$ be an $\mathfrak{m}$-primary ideal of $R$. 
Suppose $I$ is either a trace ideal or reflexive as an $R$-module. 
If $\omega_R$ has minimal multiplicity with respect to $I$, then $I$ is a stable ideal.
\end{thm}

\begin{proof} Assume $\omega_R$ has minimal multiplicity with respect to $I$, and pick a minimal reduction $(x)$ of $I$. Using Proposition~\ref{LAGAprop} with $M=N=\omega_R$,  we see that $\Hom_R(I\omega_R, \omega_R)$ is $I$-Ulrich. Thus, $(\ref{AGchar0}.1)$ implies that $\left(R:_{\rm Q}I\right)$ is $I$-Ulrich, and so $x(R:_{\rmQ}I)=I(R:_{\rmQ} I)$; see Remark~\ref{02}. 

Suppose $I$ is a trace ideal. Then $(I:_{\rmQ}I) = (R:_{\rmQ} I)$, and hence the following equalities hold:
\[
x(I:_{\rmQ}I) = x(R:_{\rmQ} I)=I(R:_{\rmQ}I) = I(I:_{\rmQ}I)=I.
\]
Here, the fourth equality is due to~\ref{fractional}(i). Therefore, $I$ is stable; see~\ref{fractional}(iii).

Next suppose $I$ is reflexive, that is, $\left(R:_{\rmQ}(R:_{\rmQ}I)\right)=I$. The following equalities hold:
\[
x^{-1}I=x^{-1} \big(R:_{\rmQ}(R:_{\rmQ}I)\big)=\big(R:_{\rmQ}x(R:_{\rmQ}I)\big) =\big(R:_{\rmQ}I(R:_{\rmQ}I)\big)=\big( \left(R:_{\rmQ}(R:_{\rmQ}I)\right):_{\rmQ} I \big) = (I :_{\rmQ} I).
\]
Here, the first and fifth equalities hold since $I$ is reflexive; the second and fourth equalities follow from the definition of colon ideals; and the third equality holds because $\left(R:_{\rm Q}I\right)$ is $I$-Ulrich, as noted earlier.
So, we have that $x^{-1} I =\left(I:_{\rmQ} I\right)$, that is, $I=x\left(I:_{\rmQ} I\right)$. This implies that $I$ is stable; see~\ref{fractional}(iii).
\end{proof}

\begin{cor} \label{corcharI} Let $R$ be a generically Gorenstein non-Gorenstein ring with canonical ideal $\omega_R$, and let $I$ be an $\fm$-primary ideal of $R$. 
Suppose $I$ is either a trace ideal or reflexive as an $R$-module. Then the following conditions are equivalent:
\begin{enumerate}[\rm(i)]
\item $I$ is stable.
\item Every torsion-free $R$-module has minimal multiplicity with respect to $I$.
\item $\omega_R$ has minimal multiplicity with respect to $I$.
\end{enumerate}
\end{cor}

\begin{proof} The implication (i) $\Longrightarrow$ (ii) follows from Remark~\ref{SI}(i), and the one (ii) $\Longrightarrow$ (iii) is clear. Moreover, Theorem~\ref{corchar1} yields (iii) $\Longrightarrow$ (i).
\end{proof}

The following special case of Corollary~\ref{corcharI} yields a new characterization of rings of minimal multiplicity.

\begin{cor} \label{corchar2} Let $R$ be a generically Gorenstein non-Gorenstein ring with canonical ideal $\omega_R$. Then the following conditions are equivalent:
\begin{enumerate}[\rm(i)]
\item $R$ has minimal multiplicity. 
\item Every torsion-free $R$-module has minimal multiplicity.
\item $\omega_R$ has minimal multiplicity.
\end{enumerate}
\end{cor}

\begin{proof} It suffices to observe that the implication (iii) $\Longrightarrow$ (i) holds; see Remark~\ref{SI}(ii). Recall that $\fm$ is a reflexive $R$-module; see~\ref{maxref}. Thus, if $\omega_R$ has minimal multiplicity, then by Corollary~\ref{corcharI} applied to the case $I=\fm$, we see that $\fm$ is stable. However, the stability of $\fm$ is equivalent to $R$ having minimal multiplicity; see Remark~\ref{SI}(ii).
\end{proof}

Section~\ref{section-MMB} investigates the case of $\fm$-primary ideals $I$ having minimal multiplicity with respect to $\omega_R$. In light of this, it is worth noting, however, that this condition, if it holds, does not imply that $\omega_R$ has minimal multiplicity with respect to $I$.

\begin{rmk} Let $R=k[\![t^4, t^5, t^7]\!]$. Then $R$ is an almost Gorenstein ring; see~\cite[4.7]{Arf1}. Thus, by Corollary~\ref{AlmGor-cor}, $\fm$ has minimal multiplicity with respect to $\omega_R$. On the other hand, $\omega_R$ does not have minimal multiplicity  (with respect to $\fm$) since $R$ does not have minimal multiplicity; see Corollary~\ref{corchar2}.
\end{rmk}

\appendix
\section{An exact sequence involving modules of minimal multiplicity} \label{apa}
The main purpose of this appendix is to prove the existence of a certain short exact sequence induced by modules of minimal multiplicity; see~\ref{Aseq}(1). We illustrate two applications of this exact sequence. First, it allows us to provide a proof of~\ref{corchar} from Section~\ref{D-M}, restated here as~\ref{Aseq}(2) for the reader’s convenience. Second, our argument can be used to establish results on the vanishing of $\Tor_i^R(M,N)$ when $M$ is a module of minimal multiplicity; such a demonstration is provided in~\ref{cx}.

Recall that we assume $R$ is a one-dimensional Cohen-Macaulay local ring, and all $R$-modules are assumed to be finitely generated. We also assume that minimal reductions of $\fm$-primary ideals of $R$ exist; see~\ref{setup}.

\begin{chunk} \label{Aseq} Let $M$ be a nonzero torsion-free $R$-module. Set $e=\e_R(M)$ and $v=\mu_R(M)$.
\begin{enumerate}[\rm(1)]
\item Assume $M$ has minimal multiplicity (with respect to $\fm$). Then $e=\mu_R(\fm M)$. If, in addition, $(x)$ is a minimal reduction of $\fm$, then there exists an exact sequence of $R$-modules
\[
0 \to k^{\oplus (e-v)} \to M/xM \to k^{\oplus v} \to 0.
\]
\item The following conditions are equivalent:
\begin{enumerate}[\rm(i)]
\item $\fm^i M$ has minimal multiplicity for all $i\geq 0$.
\item $M$ has minimal multiplicity.
\item $e=\mu_R(\fm M)$. 
\item $\fm^2 M \cong \fm M$.
\end{enumerate}
\end{enumerate}
\end{chunk}

\begin{proof} (1) Pick a minimal reduction $(x)$ of $\fm$, and consider the following commutative diagram: 
\[
\xymatrix{
0 \ar[r] & \fm M \ar[r] \ar[d]^x &M \ar[r] \ar[d]^x & M/\fm M \ar[r] \ar[d]^{x} & 0\\
0 \ar[r] & \fm M \ar[r]  & M \ar[r] & M/\fm M \ar[r] & 0 \\
}
\]
As the multiplication by $x$ on $M/\fm M$ is the zero map, the snake lemma applied to the diagram yields the following exact sequence:
\begin{equation} \label{Aseq1}\tag{\ref*{Aseq}.1}
0 \to M/\fm M \to \fm M/x\fm M \xrightarrow{\varphi} M/xM \to M/\fm M \to 0.
\end{equation}

As we assume $M$ has minimal multiplicity, it follows that $x \fm M=\fm^2 M$; see \ref{mm-defn}. Therefore, $\fm M/x\fm M \cong k^{\oplus \mu_R(\fm M)}$. Thus, since $M/\fm M \cong k^{\oplus v}$, \eqref{Aseq1} implies that $\im(\varphi) \cong k^{\oplus \big(\mu_R(\fm M)-v\big)}$. Hence, \eqref{Aseq1} yields the following short exact sequence
\begin{equation} \label{Aseq2}\tag{\ref*{Aseq}.2}
0 \to k^{\bigoplus \big(\mu_R(\fm M)-v\big)} \to M/xM \to k^{\oplus v}\to 0,
\end{equation}
which shows that $\len_R(M/xM)=\mu_R(\fm M)$. Moreover, we know $e=\len_R(M/xM)$. 
Consequently, we obtain $e=\mu_R(\fm M)$ so that \eqref{Aseq2}  is the exact sequence we seek.

(2) It follows from~\eqref{mA1} and \eqref{mA2} in Remark~\ref{mA} that the conditions in parts (i), (ii), and (iv) are equivalent. Moreover, if $M$ has minimal multiplicity, then Part~(1) yields the equality $e=\mu_R(\fm M)$. We therefore assume $e=\mu_R(\fm M)$ and proceed to show that $M$ has minimal multiplicity.

Let $(y)$ be a minimal reduction of $\fm$. Then $y$ is a non zero-divisor on $R$, and hence is a non zero-divisor on the torsion-free $R$-modules $M$ and $\fm M$. Thus, we have the following commutative diagram: 
\[
\xymatrix{
0 \ar[r] & \fm M \ar[r] \ar[d]_{\cong}^y &M \ar[r] \ar[d]_{\cong}^y & M/\fm M \ar[r] \ar[d]^{y} & 0\\
0 \ar[r] & y \fm M \ar[r]  & yM \ar[r] & yM/y \fm M \ar[r] & 0 \\
}
\]
This diagram implies, by the Snake Lemma, that $M/\fm M \cong yM/y \fm M$. So,
\begin{equation} \label{Aseq3}\tag{\ref*{Aseq}.3}
\len_R(M/\fm M) = \len_R(yM/y \fm M).
\end{equation}
Part (1) shows the existence of an exact sequence $0 \to k^{\oplus (e-v)} \to M/yM \to k^{\oplus v} \to 0$, which yields, by the additivity of length, that $\len_R(M/yM)=(e-v)+v=e$. 
Thus, the assumption $e=\mu_R(\fm M)$ implies:
\begin{equation} \label{Aseq4}\tag{\ref*{Aseq}.4}
\len_R(M/yM)=e=\mu_R(\fm M)=\len_R(\fm M/\fm^2 M).
\end{equation}

Note that we have:
\begin{equation} \label{Aseq5}\tag{\ref*{Aseq}.5}
\len_R(M/y\fm M)=\len_R(yM/y \fm M)+\len_R(M/yM).
\end{equation}
We now use \eqref{Aseq3} and \eqref{Aseq4}, and deduce from \eqref{Aseq5} that
\begin{equation} \label{Aseq6}\tag{\ref*{Aseq}.6}
\len_R(M/y\fm M)=\len_R(M/\fm M)+\len_R(\fm M/ \fm^2 M)=\len_R(M/\fm^2 M).
\end{equation}

We also have the following equality, once again by the additivity of length:
\[
\len_R(M/y\fm M)=\len_R(M/\fm^2 M)+\len_R(\fm^2M/y \fm M). 
\]
Therefore, we use \eqref{Aseq6} and deduce that $\len_R(\fm^2M/y \fm M)=0$, that is, $\fm^2 M=y \fm M$. Consequently, $M$ has minimal multiplicity; see \ref{mm-defn}.
\end{proof}

The vanishing of $\Ext$ and $\Tor$ functors over rings of arbitrary dimension, involving modules of minimal multiplicity, is studied in detail in~\cite{Souvik2}. For example, if $M$ is a nonzero torsion-free $R$-module with minimal multiplicity such that $\e_R(M) > 2\mu_R(M)$, or equivalently $\mu_R(\fm M) = \e_R(M) > 2\mu_R(M)$, and $N$ is an $R$-module, then $\Ext_R^{i}(N,M) = 0$ for all $i \gg 0$ if and only if $\pd_R(N)< \infty$; see~\cite[Thm.~B(1)(iii)]{Souvik2}. The short exact sequence constructed in~\ref{Aseq}(1) can be used to obtain similar results. Although it has potential applications in a variety of contexts beyond those considered here, in this appendix we focus on a particular instance: we examine the vanishing of $\Tor_i^R(M,N)$ for the case where $\mu_R(\fm M) = \e_R(M) > 2\mu_R(M)$, providing a brief application of~\ref{Aseq}(1) motivated by~\cite{Souvik2}.

In the following, $\beta^R_n(N)$ denotes the $nth$ Betti number of a given $R$-module $N$.

\begin{chunk} \label{cx} Let $M$ be a nonzero torsion-free $R$-module. Assume $M$ has minimal multiplicity and $\e_R(M)>2\mu_R(M)$. If there exists an $R$-module $N$ such that $\Tor^R_{i} (M,N)=0$ for all $i\gg 0$, then either $\pd_R(N)<\infty$, or $\beta^R_{i+1}(N)>\beta^R_i(N)$ for all $i\gg 0$.
\end{chunk}

\begin{proof} Assume there exists an $R$-module $N$ such that $\Tor_i^R(M,N)=0$ for all $i \gg 0$. 
Let $(x)$ be a minimal reduction of $\fm$. Then $x$ is a non zero-divisor on $R$, and hence is a non zero-divisor on $M$ 
since $M$ is a nonzero torsion-free $R$-module. Thus, there is an exact sequence of $R$-modules
\[
0 \to M \xrightarrow{x} M \to M/xM \to 0,
\]
which yields $\Tor^R_{i}(M/xM, N)=0$ for all $i\gg 0$. Set $e=\e_R(M)$, $v=\mu_R(M)$, and consider the following exact sequence obtained from~\ref{Aseq}(1):
\begin{equation} \label{cxx}\tag{\ref*{cx}.1}
0 \to k^{\oplus (e-v)} \to M/xM \to k^{\oplus v} \to 0.
\end{equation}

Suppose $\pd_R(N)=\infty$. Then $\Tor^R_{i}(k, N)\neq 0$ for all $i\geq 0$. So, tensoring the exact sequence \eqref{cxx} with $N$, we obtain the following isomorphism of nonzero $R$-modules:
\begin{equation} \label{cxxx}\tag{\ref*{cx}.2}
\Tor^R_{i}(k, N)^{\oplus (e-v)} \cong \Tor^R_{i+1}(k, N)^{\oplus v} \text{ for all } i\gg 0.
\end{equation}
These isomorphisms \eqref{cxxx} give the equality $(e-v) \cdot \beta^R_i(N)=v \cdot \beta^R_{i+1}(N)$ for all $i\gg 0$. Since $e>2v$, it follows that
\[
\dfrac{\beta^R_{i+1}(N)}{\beta^R_i(N)} = \dfrac{e-v}{v} >\dfrac{2v-v}{v}=1. 
\]
This shows that $\beta^R_{i+1}(N)>\beta^R_i(N)$ for all $i\gg 0$, as claimed.
\end{proof}

We conclude the section by noting that the hypothesis $\e_R(M)>2\mu_R(M)$ is necessary in~\ref{cx}; see also~\cite[Thm. A(1)(i,ii)]{Souvik2}.

\begin{eg} Let $R=k[\![x,y,z]\!]/(xz-y^2, xy-z^2)$. Then $R$ is a one-dimensional complete intersection of codimension two. Let $T=R/(x,y)$, $M=\Omega_R T=(x,y) \subseteq R$, and let $N=R/(x,z)$. Then $\fm M=(x^2, z^2, xz, zy)$ and hence $\mu_R(\fm M)=4$. Also, $(x)$ is a minimal reduction of $\fm$ since $\fm^3=(x)\fm^2$. So, $\e(R)=\len_R(R/xR)=4$. Thus, $4=2\mu_R(M)=\e_R(R)=\e(M)=\mu_R(\fm M)$. This shows that $M$ has minimal multiplicity; see~\ref{corchar}. Also, it follows that $\Tor^R_{i}(M, N)=0$ for all $i\geq 1$, where $\pd_R(N)=\infty$ and $\beta^R_{i}(N)=2$ for all $i\geq 1$; see~\cite[1.3]{Jo2} and~\cite[4.2]{Jo1}.  
\end{eg}

\section*{acknowledgements} Part of this work was completed during a visit by Kumashiro to West Virginia University in August 2023.

Part of this work was also completed during visits by O.~Celikbas and E.~Celikbas to the Max Planck Institute for Mathematics in Bonn, Germany from January to June 2025. O.~Celikbas gratefully acknowledges the institute’s hospitality and financial support, and E.~Celikbas likewise thanks the institute for its hospitality.

Naoki Endo was supported by JSPS Grant-in-Aid for Scientific Research (C) 23K03058. Shinya Kumashiro was supported by JSPS Grant-in-Aid for Early-Career Scientists 24K16909.

\end{document}